\documentclass[english,12pt]{amsart}
\usepackage[english]{babel}
\usepackage{amsmath,amssymb,enumerate,amsthm}
\usepackage[latin1]{inputenc}
\usepackage[T1]{fontenc}
\usepackage{cite}
\usepackage{latexsym}
\usepackage{setspace}
\usepackage{bm} 
\usepackage{comment}

\newtheorem{lemma}{Lemma}[section]
\newtheorem{theorem}[lemma]{Theorem}

\newtheorem{coro}[lemma]{Corollary}
\newtheorem{definition}[lemma]{Definition}

\newcommand{\Bad}{\mathrm{Bad}}
\newcommand{\Tad}{\mathrm{Bad}_{\Theta}}
\newcommand{\bi}{\textbf{k}}

\def\={\;=\;}
\def\>{\;>\;}
\def\<{\;<\;}
\def\:{\,:\,}
\def\.={\;\dot{=}\;}

\newcommand{\R}{\mathbb{R}}
\newcommand{\Q}{\mathbb{Q}}
\newcommand{\Z}{\mathbb{Z}}
\newcommand{\N}{\mathbb{N}}
\newcommand{\M}{\mathcal{M}}
\newcommand{\C}{\mathcal{C}}
\newcommand{\A}{\mathcal{A}}
\newcommand{\LL}{\mathcal{L}}

\newcommand{\U}{\Lambda}
\newcommand{\Ps}{\mathcal{P}_s}

\begin{document}

\title{A note on badly approximable linear forms on manifolds}

\author[P. BENGOECHEA]{Paloma BENGOECHEA$^\dag$}
\thanks{$^\dag$ Research supported by EPSRC Programme Grant: EP/J018260/1}
\address{Department of Mathematics, ETH Zurich, R\"{a}mistrasse 101, 8092 Z\"{u}rich, Switzerland}
\email{paloma.bengoechea@math.ethz.ch}

\author[N. MOSHCHEVITIN]{Nikolay MOSHCHEVITIN$^*$}
\thanks{$^*$ Research supported by RFBR Grant No. 15-01-05700a}
\address{Department of Mathematics and Mechanics, Moscow State University, Leninskie Gory 1, GZ MGU, 119991 Moscow, Russia}
\email{moshchevitin@rambler.ru}

\author[N. STEPANOVA]{Natalia STEPANOVA$^*$}
\address{Department of Mathematics and Mechanics, Moscow State University, Leninskie Gory 1, GZ MGU, 119991 Moscow, Russia}
\email{natalia.stepanova.msu@gmail.com}

\begin{abstract}
This paper is motivated by Davenport's problem and the subsequent work regarding badly approximable points in submanifolds of a Euclidian space. We study the problem in the area of twisted Diophantine approximation and present two different approaches. The first approach shows that, under a certain restriction, any countable intersection of the sets of weighted badly approximable points on any non-degenerate $\C^1$ submanifold of $\R^n$ has full dimension. In the second approach we introduce the property of isotropically winning and show that the sets of weighted badly approximable points are isotropically winning under the same restriction as above.

\end{abstract}

\maketitle

\section{Introduction}

In~\cite{KH} Khintchine proved that there exists an absolute constant $\gamma>~0$ such that for any  $\theta\in\R$ there exists $x\in\R$ satisfying
\begin{equation}\label{K1}
\inf\limits_{q \in \mathbb{N}} q \cdot \| q\theta - x\| \ge \gamma.
 \end{equation}
The best known value of $\gamma$ is probably due to H. Godwin~\cite{HG}. More than 50 years later, Tseng \cite{T} showed that,
 for every $\theta \in \mathbb{R}$, the set of all $x$ for which there exists a positive constant $\gamma=\gamma(\theta,x)$ such that~\eqref{K1} is true is $1/8$-winning for the standard Schmidt game (in particular it has maximal Hausdorff dimension). 
Now we usually refer to such sets, denoted by $\Bad_\theta$ as sets of \textit{twisted badly approximable numbers}.

In the other direction, Kim \cite{K} proves that $\Bad_\theta$ has 0 Lebesgue measure if $\theta\in\R\backslash\Q$. 
 
The study of twisted badly approximable numbers has been pursued to higher dimension as the classical approximation by rationnals.
Various multidimensional generalizations of Khinchine's result were due to Khinchine~\cite{KH2, KH3}, Jarn\'{i}k \cite{J,J2}, Kleinbock~\cite{KL}, Bugeaud, Harrap, Kristensen, Velani~\cite{BHKV}, Moschevitin~\cite{NGM}, Einsiedler and Tseng~\cite{ET} and others.
In the classical case, Schmidt introduced the weighted simultaneously badly approximable numbers 
$$
\Bad(i,j)=\left\{(x_1,x_2)\in\R^2:\, \inf_{q\in\N}\max(q^i\|qx_1\|, q^j\|qx_2\|)>0\right\},
$$
where $i,j$ are real positive numbers satisfying $i+j=1$.

It is well-known that $\Bad(i,j)$ has Lebesgue measure 0 and it was shown to have full Hausdorff dimension by Pollington and Velani \cite{PV} only in 2002. A few years later, Badziahin, Pollington and Velani made a breakthrough \cite{BPV} by settling a famous conjecture of Schmidt. They proved that any countable intersection with a certain restriction of sets $\Bad(i,j)$ for different weights $(i,j)$ has full Hausdorff dimension. This result was significantly improved by An in \cite{An}, where he proved that $\Bad(i,j)$ is winning for the standard Schmidt game, so the above mentioned result remains true if we replace `finite' by `countable'. 

Moreover, Badziahin, Pollington and Velani's work settled the foundations for the study of the dimension of $\Bad(i,j)$ on planar curves. 
In 1964, Davenport asked the question: is the intersection of $\Bad(\frac{1}{2},\frac{1}{2})$ with the parabola uncountable? Badziahin and Velani \cite{BV} answered positively this question and proved the more general result: the set $\Bad(i,j)$ on any $C^{(2)}$ non-degenerate planar curve has full dimension. This is false in general if we replace non-degenerate curves by straight lines. 
Recently Badziahin and Velani's result has been improved to winning in \cite{ABV} and generalized to higher dimension by Beresnevich \cite{Ber}. 
In higher dimension, we fix an $n$-tuple $\bi=(k_1,\ldots,k_n)$ of real numbers satisfying
\begin{equation}\label{weight}
k_1,\ldots,k_n>0\qquad\mbox{and}\qquad \sum_{i=1}^n k_i=1,
\end{equation}
and define
$$
\Bad(\bi,n,m)=\left\{\Theta\in\mathrm{Mat}_{n\times m}(\R):\, \inf_{q\in\Z^m_{\neq 0}}\max_{1\leq i\leq n}(|q|^{mk_i}\|\Theta_i(q)\|)>0\right\}.
$$
Here, $|\cdot|$ denotes the supremum norm, $\Theta=(\Theta_{ij})$ and $\Theta_i(q)$ is the product of the $i$-th line of $\Theta$ with the vector $q$, i.e.
$$
\Theta_i(q)=\sum^m_{j=1} q_j \Theta_{ij}.
$$ 
Beresnevich proved that any countable intersection with a certain restriction of sets $\Bad(\bi,n,1)$ for different weights $\bi$ (and same dimension of approximation $n\times 1$) with any analytic non-degenerate manifold in $\R^n$ has full dimension. Thus he also settled a generalization to higher dimension of Schmidt's conjecture. 

In twisted Diophantine approximation, much less is known. Given $\Theta\in\mathrm{Mat}_{n\times m}(\R)$, we define
$$
\Tad(\bi,n,m)=\left\{x\in\R^n:\, \inf_{q\in\Z^m_{\neq 0}}\max_{1\leq i\leq n}(|q|^{mk_i}\|\Theta_i(q)-x_i\|)>0\right\}.
$$
The Lebesgue measure of $\Tad(\bi,n,1)$ is zero for almost all $\Theta\in\R^n$ and its Hausdorff dimension is always maximal  (see \cite{BM}). The winning property in $n\times m$ dimension has been proved by Harrap and Moshchevitin \cite{HM} provided that $\Theta\in\Bad(\bi,n,m)$. The general problem remains open even for $n=2, m=1$.

In this note, we keep the condition $\Theta\in\Bad(\bi,n,m)$ and study the Hausdorff dimension and winning property of $\Tad(\bi,n,m)$ on manifolds of $\R^n$.

\section{Statement and discussion of results}

There are different definitions for a set to be winning in a Euclidian space $\mathcal{E}$. In this paper we consider classical Schmidt's definition. So when we say that a set is winning in the Euclidean space $\mathcal{E}$ we mean that it is $\alpha$-winning in $\mathcal{E}$ for some $0<\alpha\leq\frac{1}{2}$ in the sense of Schmidt's game \cite{Sch1}, \cite{Sch2}, \cite{Sch3}.
One of the reasons is that we apply a lemma from \cite{NGM} that was formulated for classical Schmidt's games. However, it is clear that Theorems \ref{th curve} and \ref{Affine} will be true for the hyperplane absolute game (for the definition see \cite{FSU}).

A $C^{(1)}$ curve $\C$ in $\R^n$ can be written in the form
\begin{equation}\label{def curve}
\C=\left\{(f_1(x),\ldots,f_n(x)),\ x\in I\right\}
\end{equation}
where $f=(f_1,\ldots,f_n):I\rightarrow\R^n$ is a $C^{(1)}$ map defined on a compact interval $I\subset\R$. We call $\C$ \textit{non-degenerate} (respectively \textit{everywhere non-degenerate}) if for all $i=1,\dots,n$ we have $f'_i(x)\neq 0$ for some $x\in I$ (resp. for all $x\in I$). We call a  $C^{(1)}$ submanifold of $\R^n$ \textit{non-degenerate} if it can be foliated by non-degenerate $C^{(1)}$ curves.

\begin{theorem}\label{th curve}
Let $\bi=(k_1,\ldots,k_n)$  be an $n$-tuple of positive real numbers satisfying $\sum_{i=1}^n k_i=1$. 
Let $\Theta\in\Bad(\bi,n,m)$ and $\C\subset\R^n$ be an everywhere non-degenerate $C^{(1)}$ curve. Then the set
$\Tad(\bi,n,m)\cap \C$ is winning.
\end{theorem}
 
In view of the fact that the intersection
of countably many winning sets is winning and that a winning set has full dimension, given any countable collection $(\bi_t)_{t\geq 1}$ of $n$-tuples $\bi_t=(k_{1,t},\ldots,k_{n,t})$ of positive 
real numbers satisfying $\sum_{i=1}^n k_{i,t}=1$, Theorem \ref{th curve} implies that 
\begin{equation}\label{dim curve}
\dim\Big(\bigcap_{t\geq 1}\Tad(\bi_t,n,m)\cap \C\Big)=1
\end{equation} 
for any $\Theta\in\bigcap_{t\geq 1}\Bad(\bi_t,n,m)$ and any everywhere non-degenerate $C^{(1)}$ curve $\C\subset\R^n$. If we require that $\C$ is non-degenerate only for some $x\in I$ instead of everywhere,  we can still choose an interval $I_0\subset I$ small enough so that the shorter curve $\C^\ast=\left\{(f_1(x),\ldots,f_n(x)),\ x\in I_0\right\}$
 clearly is everywhere non-degenerate. Then we have equation \eqref{dim curve} for $\C^\ast$ and of course also for $C$.
The fibering technique (see \cite{Spr} p. 9-10, or also \cite{Ber} section 2.1) establishes the generalization of \eqref{dim curve} to non-degenerate $C^{(1)}$ manifolds.

\begin{coro}\label{th manifolds}
Let $(\bi_t)_{t\geq 1}$ be a countable collection of $n$-tuples $\bi_t=(k_{1,t},\ldots,k_{n,t})$ of positive 
real numbers satisfying $\sum_{i=1}^n k_{i,t}=1$.
Let $\Theta\in\bigcap_{t\geq 1}\Bad(\bi_t,n,m)$ and $\M\subset\R^n$ be a non-degenerate $C^{(1)}$ manifold. Then 
$$
\dim\Big(\bigcap_{t\geq 1}\Tad(\bi_t,n,m)\cap \M\Big)=\dim(\M).
$$ 
\end{coro}

The proofs of Theorem \ref{th curve} and Corollary \ref{th manifolds} are given in Section 3. The arguments used there do not provide winning on everywhere non-degenerate $C^{(1)}$ manifolds unless the manifold is a curve.
 However, a different strategy can be adopted to prove winning on affine subspaces.

\begin{definition}
We call a set $\mathcal{N} \subset \mathbb{R}^n$ \it{isotropically winning} 
if for each $d \le n$ and for each $d$-dimensional affine subspace $\mathcal{A} \subset \mathbb{R}^n$ the intersection $\mathcal{N} \cap \mathcal{A}$ is $1/2$-winning.
\end{definition}
 
Here we should note that the isotropically winning property is really a very strong property. For example, the set $\Bad(\frac{1}{2},\frac{1}{2})$ is $1/2$-winning in Schmidt's sense, however it is not isotropically winning.

\begin{theorem}\label{Affine}
If  $\Theta \in \Bad(\bm{k}, n,m)$, then $\Bad_{\Theta} (\bm{k}, n,m)$ is isotropically winning.
\end{theorem}
Section 4 is devoted to the proof of Theorem \ref{Affine}. We give an outline of the proof in the subsection 4.1.

\section{Proof of Theorem \ref{th curve}}

 Let $\C\subset\R^n$ be an everywhere non-degenerate $C^{(1)}$ curve as in \eqref{def curve}.
Let $\bi=(k_1,\ldots,k_n)$  be an $n$-tuple of real numbers satisfying \eqref{weight}.
Without loss of generality assume that $k_1=\max_{1\leq i\leq n}(k_i)$. Let $\pi:\R^n\rightarrow\R$ be the projection map onto the first coordinate. 

\begin{theorem}\label{th projected}
If $\Theta\in\Bad(\bi,n,m)$, the set $\pi(\Tad(\bi,n,m)\cap \C)$ is 1/4-winning.
\end{theorem} 

If $\max_{1\leq i \leq n}(k_i)=k_t\neq k_1$, then we consider the projection onto the $t$-th coordinate rather than the first. Since the projection map $\pi$ is bi-Lipschitz on $\C$ and the image of a winning set under a bi-Lipschitz map is again winning (see \cite{D} Proposition 5.3), Theorem \ref{th curve} follows from Theorem \ref{th projected}.
\\

\textbf{Proof.}
Let $\kappa\geq 1$ be a constant such that
\begin{equation}\label{kappa}
|f_i(x)-f_i(x')|\leq\kappa|f_1(x)-f_1(x')| \qquad (i=1,\ldots,n)
\end{equation}
for all $x,x'\in I$. 
 
Since $\Theta\in\Bad(\bi,n,m)$, there exists a constant $0<c<1$ satisfying
\begin{equation}\label{Bad}
\max_{1\leq i\leq n}(|q|^{mk_i}\|\Theta_i(q)\|)>c\qquad\forall q\in \mathbb{Z}^m_{\neq 0}.
\end{equation} 

Alice and Bob play a Schmidt game on the interval $f_1(I)$. A Schmidt game involves two real numbers $\alpha,\beta\in (0,1)$ and starts with Bob choosing a closed interval $B_0\subset f_1(I)$. Next, Alice chooses a closed interval $A_0\subset B_0$ of length $\alpha|B_0|$. Then, Bob chooses at will a closed interval $B_1\subset A_0$ of length $\beta\alpha|B_0|$. Alice and Bob keep playing alternately in this way, generating a nested sequence of closed intervals in $f_1(I)$:
$$
B_0\supset A_0\supset B_1\supset A_1\supset\ldots\supset B_s\supset A_s\supset\ldots
$$
with lengths
$$
|A_s|=\alpha|B_s|\qquad\mbox{and}\qquad|B_s|=\beta |A_{s-1}|=(\alpha\beta)^s|B_0|.
$$
The subset $\pi(\Tad(\bi,n,m)\cap \C)$ is called $\alpha$-winning if Alice can play so that the unique point of intersection
$$
\bigcap_{s=0}^\infty B_s=\bigcap_{s=0}^\infty A_s
$$
lies in $\pi(\Tad(\bi,n,m)\cap \C)$ whatever the value of $\beta$ is.
The goal is to describe a $1/4$-winning strategy for Alice.
 
Suppose Bob has chosen a closed interval $B_0\subset f_1(I)$. We can assume $|B_0|$ to be of length as small as we want. In particular, we assume 
\begin{equation}\label{I}
|B_0|<\dfrac{c}{2\kappa}.
\end{equation}
Let
$$
R=\left(\dfrac{4}{\beta}\right)^{1/mk_1},
$$
and fix $\epsilon>0$ such that
\begin{equation}\label{epsilon}
\epsilon<\dfrac{|B_0|}{4R^{mk_1}}.
\end{equation}
By the definition of the game, for each $s\geq 0$, 
\begin{equation}\label{Bk}
|B_s|=R^{-smk_1}|B_0|.
\end{equation}

For each $(p,q)\in\Z^n\times\Z^m$, let 
$$
\Delta(p,q)=\left\{f_1(x):\, x\in I,\ \max_{1\leq i\leq n}(|q|^{mk_i}|\Theta_i(q)-f_i(x)-p_i|)<\epsilon\right\}.
$$
The interval $\Delta(p,q)$ is the projection of the points in $\C$ lying in a hyperrectangle centred at $(\Theta_1(q)-p_1,\ldots, \Theta_n(q)-p_n)$ of size $\epsilon |q|^{-mk_1}\times\ldots\times\epsilon |q|^{-mk_n}$.
Clearly
$$
f_1(I)\backslash\bigcup_{(p,q)\in\Z^n\times\Z^m}\Delta(p,q)\subset\pi(\Tad(\bi,n,m)\cap\C).
$$

We define a partition of $\Z^n\times\Z^m$ by letting
\begin{equation}\label{Pk}
\Ps=\left\{(p,q)\in\Z^n\times\Z^m : R^{s-1}< |q|\leq R^s\right\}\qquad(s\geq 0).
\end{equation}
We prove that, for every $s\geq 0$, Alice can play such that
\begin{equation}\label{k}
A_s\subset f_1(I)\backslash\bigcup_{(p,q)\in\Ps}\Delta(p,q).
\end{equation}
Thus we will have
\begin{equation}\label{interseccion}
\bigcap_{s=0}^\infty B_s=\bigcap_{s=0}^\infty A_s\subset \pi(\Tad(\bi,n,m)\cap\C),
\end{equation}
and this will prove the theorem.
\\

\textbf{Fact 1.} For each $(p,q)\in\Ps$, we have that
$$
|\Delta(p,q)|\leq\dfrac{2\epsilon}{|q|^{mk_1}} \stackrel{\eqref{Pk}}{<} \dfrac{2\epsilon}{R^{(s-1)mk_1}} \stackrel{\eqref{epsilon}}{<}\dfrac{|B_0|}{2R^{smk_1}}\stackrel{\eqref{Bk}}{=}\frac{1}{2}|B_{s}|.
$$

\textbf{Fact 2.} Suppose $(p,q)$ and $(p',q')$ are two points in $\mathcal{P}_s$ such that $\Delta(p,q)\cap B_s\neq\emptyset$ and $\Delta(p',q')\cap B_s\neq\emptyset$, i.e. suppose that there exist $x,x'\in I$ such that for all $i=1,\ldots,n$,
$$
|\Theta_i(q)-f_i(x)-p_i|<\dfrac{\epsilon}{|q|^{mk_i}},\qquad |\Theta_i(q')-f_i(x')-p'_i|<\dfrac{\epsilon}{|q'|^{mk_i}}
$$
and
$$
|f_1(x)-f_1(x')|\leq|B_s|.
$$
Next we show that $(p,q)=(p',q')$. Indeed, the inequality  \eqref{kappa} implies that, for all $i=1,\ldots,n$,
\begin{align*}
|\Theta_i(q-q')-(p_i-p'_i)|\leq |\Theta_i(q)-f_i(x)-p_i|+&|\Theta_i(q')-f_i(x')-p'_i|\\
&+\kappa|f_1(x)-f_1(x')|.
\end{align*}
Hence
\begin{align*}
|\Theta_i(q-q')-(p_i-p'_i)|
&\stackrel{\eqref{Pk},\eqref{Bk}}{<}\dfrac{2\epsilon}{R^{(s-1)mk_i}}+\dfrac{\kappa|B_0|}{R^{smk_1}}\\
&\stackrel{\eqref{epsilon},\eqref{I}}{<}\dfrac{c}{2R^{smk_i}}+\dfrac{c}{2R^{smk_i}}\\
&\stackrel{\eqref{Pk}}{\leq}\dfrac{c}{|q-q'|^{mk_i}}
\end{align*}
if $q\neq q'$. 
Now the condition \eqref{Bad} implies that the last inequality cannot hold for all $i=1,\ldots,n$ and hence $q=q'$. Then, from the second-to-last inequality it follows that, for all $i=1,\ldots,n$,
$$
|p_i-p'_i|<1,
$$
so $p_i=p_i'$. Thus we conclude that $p=p'$.
\\

A straightforward consequence of the above two facts is that Alice can choose an interval $A_s\subset B_s$ of length $\frac{1}{4}|B_s|$ that avoids $\Delta(p,q)$ for all $(p,q)\in\Ps$. This completes the proof of \eqref{k}. 
\begin{flushright}
$\square$
\end{flushright}

\section{Proof of Theorem \ref{Affine}}

Our exposition is organized as follows.
In the subsection 4.1 we briefly explain the strategy of the proof.
In the subsections~4.2 and ~4.6 we use transference arguments.
In the subsections 4.3, 4.4, 4.5 we explore the geometry of 
auxiliary subspaces. 

Throughout this section we denote by $|\cdot|_e$ the Euclidian norm.

\subsection{Outline of the proof}

Let $1\le d \le n$. Let $\mathcal{A}$ be a $d$-dimensional affine subspace  
and
$\LL$ be the corresponding $d$-dimensional linear subspace (so $\LL$ is the translation of $\A$ which contains the origin).

We briefly explain the main construction of our proof.
 We construct a special 
 sequence $\Lambda$ of integer vectors
 $\bm{ u}_r =\bm{ u}_r (\mathcal{L}) = (u_{r,1},\ldots,u_{r,n}) \in \mathbb{Z}^n$
 which is useful for Khintchine's type of inhomogeneous transference argument (see Ch. V from Cassels's book \cite{JC}). 
  For the sequence $\Lambda$ of integer vectors $\bm{ u}_r$  
we define the set
\begin{equation*}
\aligned
N(\U)& = \{ \bm{x} \in \mathcal{A} :\, \inf_{r\geq 1}\|x_1 u_{r,1} +\ldots +x_n u_{r,n} \| >0\}.
 \endaligned
\end{equation*}
To establish the result of Theorem \ref{Affine} it is enough to prove the following two facts:

{\bf Fact A.}
 $N(\U) \subset \Bad_{\Theta} (\bm{k}, n,m)$,\label{fact1}

 {\bf Fact B.}
$N(\U)$ is winning.\label{fact2}

{\bf Fact A} will follow from the inhomogeneous transference argument; 
we give a detailed exposition in the subsection 4.6. {\bf Fact B} will follow from the construction of $\Lambda$ 
and from Lemma 1 from \cite{NGM}. 
This lemma establishes the winning property of any set of 
the form $N(\U)$ in the special case when $d=n$ and so $\A=\R^n$, 
under the condition that the Euclidian norms of the elements of the sequence $\Lambda$ are lacunary.
Only a minor modification is needed in order to deduce the winning property for the intersection $N(\U) \cap \mathcal{A}$ 
for an arbitrary $d$-dimensional subspace $\A\subset\R^n$, under the condition that the Euclidian norms of the projections 
$\bm{u}^\LL_r$ of $\bm{u}_r$ onto $\LL$ are lacunary, i.e.
\begin{equation}\label{1}
\dfrac{|\bm{u}^\LL_{r+1}|_e} {|\bm{u}^\LL_r|_e} \ge M, \qquad r= 1,2,3, \ldots
\end{equation}
for some $M > 1$.
So, as soon as one checks condition \eqref{1} then Lemma~1 from \cite{NGM} automatically establishes {\bf Fact B} for the set $N(\U)$.

\subsection{Dual setting}

The condition $\Theta\in\Bad(\bi,n,m)$ is the key to ensure that we are able to construct a sequence $\Lambda$ satisfying \eqref{1}. 
That condition has the following dual reformulation (see \cite{BPV} Appendix, Theorem 6, for the proof) in terms of the transposed matrix $\Theta^*$: 
\begin{equation}\label{bad cond}
\Theta\in\Bad(\bi,n,m)\Leftrightarrow\inf_{\bm{q} \in \mathbb{Z}_{\neq 0}^n}\max\limits_{1 \le i \le n} 
(|q_i|^{\frac{1}{m k_i}} ) \max\limits_{1\leq j\leq m}\|\Theta^*_j(\bm{q})\| >\gamma
\end{equation}
for some positive constant $\gamma=\gamma(\Theta)$.
In the sequel we suppose everywhere that
$$
\gamma <1.
$$

One can find generalizations of this dual reformulation in \cite{GE}.

\subsection{Subspaces}

Without loss of generality we suppose $k_1 \ge \ldots \ge k_n.$
Consider the subspaces 
$$
\Gamma_0=\R^n,\qquad \Gamma_i = \{ \bm{x}\in\R^n: x_1=\ldots= x_i = 0 \}\quad (1 \le i \le n).
$$
It is clear that  $\Gamma_i $ is an $(n-i)$-dimensional subspace of $\R^n$
and 
$$\Gamma_0 \supset \Gamma_1 \supset \Gamma_2 \supset \ldots \supset \Gamma_n.$$

Recall that $\mathcal{L} \subset \mathbb{R}^n$ is a $d$-dimensional linear subspace.
We define $t$ to be the minimal positive integer ($d-1 \le t \le n-1$) such that
$$
\mathcal{L} \subset \Gamma_{n-(t+1)} , \qquad
\mathcal{L} \not\subset \Gamma_{n-t}.
 $$
Denote by $\ell_{n-t}$ the coordinate line 
$$
\ell_{n-t}=\{ \bm{x}\in\R^n: x_i = 0\ \forall i\neq n-t\}.
$$

Since $\mathcal{L} \subset \Gamma_{n-(t+1)}$, 
the first $n-(t+1)$ coordinates of each point $\bm{x}\in\LL$ vanish, so
the set $N(\U)$ has the form
$$N(\U) = \{ \bm{x} \in \mathcal{A}  : \inf_{r\geq 1}\|u_{r,n-t}\cdot x_{n-t} + \ldots+ u_{r,n} \cdot x_n\| > 0 \}.$$
Hence when we construct the sequence $\Lambda$ satisfying \eqref{1}, we only have to take into account the last $t+1$ coordinates.

 In the rest of the proof we will deal with two spaces: the $(n+m)$-dimensional space 
 $$\mathbb{R}^{n+m} =\{ (\bm{x},\bm{y})= (x_1,\ldots,x_n,y_1,\ldots,y_m)\}
 $$
 and the $(t+1)$-dimensional subspace 
$$\mathbb{R}^{t+1}=\Gamma_{n-(t+1) } =\{ \bm{x}= (x_{n-t},\ldots,x_n)\}.$$
We  identify $\LL$, $\ell_{n-t}$ and $\Gamma_{n-t}$ with subspaces in
$\mathbb{R}^{t+1}$, so when we consider the angles between them, we are just considering the angles  in the $(t+1)$-dimensional space $\mathbb{R}^{t+1}$.
For a vector 
$\bm{u} \in \mathbb{R}^n$ 
we denote by
$\bm{\tilde{u}}$
 its projection onto the subspace $\mathbb{R}^{t+1}=\Gamma_{n-(t+1) }$.

For the clarity of the exposition we  distinguish two cases: {\bf Case 1}  and {\bf Case 2.}

{\bf Case 1.}
If $\mathcal{L} = \Gamma_{n-(t+1)}$, then $t=d-1$ and the projections  $\bm{u}^\LL_r$ are just the projections $\bm{\tilde{u}}_r$ of $\bm{u}_r$ onto $\R^{t+1}$. In this case, it is enough for the property~\eqref{1} to ensure that the sequence $|\bm{\tilde{u}}_r|_e$ is lacunary. 

{\bf Case 2.}
If $\mathcal{L} \neq \Gamma_{n-(t+1)}$, then $t\geq d\ge 1$. In this case,
we need to consider the angle $\omega = \widehat{\mathcal{L}, \ell_{n-t}}$ between the subspace
$\mathcal{L}$ and the one-dimensional $ \ell_{n-t}$. 
Here the angle between subspace $\mathcal{A}$ and one-dimensional subspace $l$  is defined as
$$\widehat{\mathcal{A}, l} = \min_{a\in \mathcal{A}\setminus\{{\bf 0}\}} \widehat{a,l}.$$
It is clear that 
$\widehat{\Gamma_{n-t}, \ell_{n-t}}=\frac{\pi}{2}$, so
as $\mathcal{L} \not\subset \Gamma_{n-t}$ we have
$$
 0\le \omega = \widehat{\mathcal{L}, \ell_{n-t}} <\frac{\pi}{2}.
 $$

Let us consider an arbitrary vector $\bm{u}\in \mathbb{R}^n$.
We should note that the Euclidean norms of the projections  
$\bm{{u}}^\LL
$
and
$
\bm{\tilde{u}}
$
satisfy the equality
$$
|\bm{{u}}^\LL|_e =
|\bm{\tilde{u}}|_e
\cos
\widehat{\mathcal{L}, 
\bm{\tilde{u}}}
.$$
So to get the lacunarity condition
(\ref{1})  with a given $M>1$ for the sequence of Euclidean norms of the projections 
$\bm{u}^\LL_r$
for a certain sequence of vectors $\bm{u}_r$
we should establish two facts:

{\bf Fact C.} the lacunarity of the sequence $|\bm{\tilde{u}}_r|_e$, that is
\begin{equation}\label{1l}
\dfrac{|\bm{\tilde{u}}_{r+1}|_e} {|\bm{\tilde{u}}_r|_e} \ge \tilde{M}, \qquad r= 1,2,3, \ldots,
\end{equation}

{\bf Fact D.} the additional condition 
\begin{equation}\label{angle}
 \widehat{\ell_{n-t},\bm{\tilde{u}}_r} =
 \frac{\pi}{2} -  \widehat{\bm{\tilde{u}}_r,\Gamma_{n-t}}
 \le \sigma, \qquad r= 1,2,3, \ldots,
\end{equation}
where 
the values $\tilde{M}$ and $\sigma$ must satisfy the inequality
\begin{equation}\label{inee}
 \tilde{M}\frac{\cos (\omega+\sigma)}{\cos (\omega - \sigma)}=M>1.
\end{equation}

Indeed,
suppose that
$ \widehat{\ell_{n-t},\bm{\tilde{u}}} < \sigma$.
Then
by the definition of the angle and the triangle inequality we have
$$
 \widehat{\bm{\tilde{u}},\LL} =
 \widehat{\bm{\tilde{u}},\bm{\tilde{u}}^\LL}\le
 \widehat{\bm{\tilde{u}},{\ell_{n-t}}^\LL}
\le
\widehat{\bm{\tilde{u}},{\ell_{n-t}}}+
\widehat{\ell_{n-t},{\ell_{n-t}}^\LL}
\le \sigma+ \omega.
$$
Of course here $\bm{\tilde{u}}^\LL=\bm{{u}}^\LL$.
Analogously
$$
 \widehat{\bm{\tilde{u}},\LL} 
+
 \widehat{\bm{\tilde{u}},\ell_{n-t}} 
=
 \widehat{\bm{\tilde{u}},\bm{\tilde{u}}^\LL} 
+
 \widehat{\bm{\tilde{u}},\ell_{n-t}} 
 \ge
  \widehat{\bm{\tilde{u}}^\LL,\ell_{n-t}}
  \ge
   \widehat{\LL,\ell_{n-t}} = \omega.
   $$
So we have
$$
 \omega-\sigma\le  \widehat{\bm{\tilde{u}},\LL} \le \omega+\sigma.
$$
Now from (\ref{1l},\ref{angle}) and (\ref{inee}) we deduce
$$
\dfrac{|\bm{u}^\LL_{r+1}|_e} {|\bm{u}^\LL_r|_e}=
\dfrac{|\bm{\tilde{u}}_{r+1}|_e \cos \widehat{\mathcal{L}, 
\bm{\tilde{u}}}_{r+1} } {|\bm{\tilde{u}}_r|_e\cos \widehat{\mathcal{L}, 
\bm{\tilde{u}}}_r} 
\ge
\dfrac{|\bm{\tilde{u}}_{r+1}|_e \cos (\omega+\sigma) } {|\bm{\tilde{u}}_r|_e\cos (\omega - \sigma)} 
\ge M>1,
$$
and this gives (\ref{1}).

Now we put $M=2$. Then $\tilde{M} =  \frac{2\cos (\omega-\sigma)}{\cos (\omega + \sigma)}$.
So to get (\ref{1l},\ref{inee}) it is enough to satisfy the condition
\begin{equation}\label{lasa}
 \dfrac{|\bm{\tilde{u}}_{r+1}|_e} {|\bm{\tilde{u}}_r|_e}\ge\frac{2\cos (\omega-\sigma)}{\cos (\omega + \sigma)}, \qquad r= 1,2,3, \ldots.
\end{equation}

In the next subsection, we construct the sequence $\Lambda$ by constructing a sequence of parallelepipeds $\Pi_T$ in $\R^{n+m}$ 
and by choosing a certain integer vector $\bm{u}_r$ in each of them. To ensure \eqref{angle}, we construct the corresponding  
parallelepipeds  $\overline{\Pi}_T$ in $\R^{t+1}$ very long in the $n-t$ direction and short
in the other directions corresponding to $n-t+j,\quad 1 \le j \le t$.
(The precise definitions will be given in the next subsection.)
Hence the vectors $\bm{u}_r$ we choose in each parallelepiped are close to the line $\ell_{n-t}$ and their projections onto $\LL$ are close to the line $\ell^\LL_{n-t}$. 

\subsection{Parallelepipeds}

 First of all we put
 $$
 \sigma = \min \left(\frac{\omega}{2},\frac{\pi}{4}-\frac{\omega}{2}\right).
 $$
 Then
 $$
  0<\sigma <\frac{\pi}{4}.
  $$
  For $t \ge 1$ we define 
\begin{equation}\label{lamed}
 \lambda = \frac{\sqrt{t}}{\tan \sigma} >1.
\end{equation}

We proceed now to the construction of the sequence $\Lambda$.
Given  $T \ge~1$ and a collection of strictly positive real numbers $\beta_1,\ldots, \beta_{n+1}$, we consider the $(n+m)$-dimensional parallelepiped
\begin{align*}
\Pi_T (\beta_1 ,\ldots, \beta_{n+1})= \{ (\bm{x}, \bm{y}) \in\R^n\times &\R^m :\, |x_i| \le \beta_i T^{mk_i} \ (1 \le i \le n),\\
  &\max\limits_{1 \le j \le m} |\Theta^*_j(\bm{x}) - y_j| 
\le \beta_{n+1} T^{-1}  \}
\end{align*}
and its projection onto $\mathbb{R}^{t+1}$
$$
\overline{\Pi}_{T} (\beta_{n-t},\ldots,\beta_n)  = \{ (x_{n-t},\ldots, x_n ) \in \mathbb{R}^{t+1} : |x_i| \le \beta_i T^{mk_i} \ (n-t
  \le i \le n)\}.
$$
By \eqref{bad cond} we have 
$$
\Pi_T(1, \ldots, 1, \gamma)\cap \mathbb{Z}^{n+m} = \{0\}.
$$
As $\lambda>1$ we have
\begin{equation}\label{lamed1}
\Pi_T(\underbrace{1, \ldots, 1}_{n-t},\underbrace{\lambda^{-1},...,\lambda^{-1}}_{t}, \gamma)\cap \mathbb{Z}^{n+m} = \{0\}.
\end{equation}



However, the parallelepiped
$\Pi_T (\underbrace{1,\ldots,1}_{n-t-1},\gamma^{-m}\lambda^t,\underbrace{\lambda^{-1},...,\lambda^{-1}}_{t}, \gamma) $
 is convex, symmetric, with volume
$$
 \gamma^{-m} \lambda^t\cdot \left(\prod\limits_{i=1}^n 2T^{mk_i} \right) \cdot 2^m \gamma^m \lambda^{-t} T^{-m} = 2^{n+m}
$$
and then, by Minkowski's Convex Body Theorem, we have
$$
\Pi_T (\underbrace{1,\ldots,1}_{n-t-1},\gamma^{-m}\lambda^t,\underbrace{\lambda^{-1},...,\lambda^{-1}}_{t}, \gamma)\cap \mathbb{Z}^{n+m} \neq \{0\}.
$$
Therefore, for each $T \ge 1$ there exists at least one integer vector $\bm{w} = (\bm{u}, \bm{v}) \in \mathbb{Z}^{n+m}$ such that
$$
\bm{w} \in \Pi_T (\underbrace{1,\ldots,1}_{n-t-1},\gamma^{-m}\lambda^t,\underbrace{\lambda^{-1},...,\lambda^{-1}}_{t}, \gamma)  \setminus
\Pi_T(\underbrace{1, \ldots, 1}_{n-t},\underbrace{\lambda^{-1},...,\lambda^{-1}}_{t}, \gamma).
$$
Among these integer vectors we choose one with the smallest coordinate $|u_{n-t}| \ge 1$. If this vector is not unique we choose that 
for which $\max\limits_{1 \le j \le m} |\Theta^*_j(\bm{u}) - v_j|$ attains its minimal value.
We denote this vector by
$$
\bm{w}(T) = (\bm{u}(T),\bm{v}(T))  = (u_1 (T), \ldots, u_n (T), v_1 (T), \ldots, v_m (T))
$$
and define
$$
\psi (T)  =\max\limits_{1 \le j \le m} \|\Theta^*_j(\bm{u}(T))\|= \max\limits_{1 \le j \le m} |\Theta^*_j(\bm{u}(T)) - v_j(T)|
.
$$

Since $\bm{w}(T) \in \Pi_T (\underbrace{1,\ldots,1}_{n-t-1},\gamma^{-m}\lambda^t, \lambda^{-1}, \ldots, 1\lambda^{-1}, \gamma)$, one has   
\begin{equation}\label{2.1''}
| u_{i}(T)| \le T^{mk_{i}}, \quad 1 \le i \le n-t-1,
\end{equation}
\begin{equation}\label{2uy}
|u_{n-t} (T)| \le \lambda^t\,\gamma^{-m} T^{mk_{n-t}},
\end{equation}
\begin{equation}\label{2uz}
 | u_{i}(T)| \le \lambda^{-1}T^{mk_{i}}, \quad n-t+1 \le i \le n.
\end{equation}
Also
we have
\begin{equation}\label{2.2''}
\psi (T) \le \gamma T^{-1}.
\end{equation}
Since $\bm{w} \notin \Pi_T(\underbrace{1, \ldots, 1}_{n-t},\underbrace{\lambda^{-1},...,\lambda^{-1}}_{t}, \gamma)$, it is clear that
\begin{equation}\label{2uu}
|u_{n-t} (T)|>T^{mk_{n-t}}
\end{equation}
and so
\begin{equation}\label{P1}
\max\limits_{1 \le i \le n} (|u_i(T)|^{1/(mk_i)}) = |u_{n-t}(T)|^{1/(mk_{n-t})}.
\end{equation}
By \eqref{bad cond} and (\ref{P1}) one has
\begin{equation}\label{P2}
\psi (T) \ge \gamma (\max\limits_{1 \le i \le n} (|u_i(T)|^{1/(mk_i)}))^{-1} \ge \lambda^{-t/mk_{n-t}}\,\gamma^{1+1/k_{n-t}} T^{-1}.
\end{equation}
Here we should note that in the {\bf Case 2} we have $t\ge 1$ and  for the projection $\bm{\tilde{u}} (T)$ of the vector $\bm{u} (T)$ onto $\mathbb{R}^{t+1}$
from the definition (\ref{lamed}), conditions (\ref{2uz},\ref{2uu}) and the condition $ k_{j} \ge k_{j+1}, \, n-t\le j \le n $ we have
\begin{equation}\label{ou}
\bm{\tilde{u}} (T)
\in 
 \overline{\Pi}_{T} (\gamma^{-m}\lambda^t ,\lambda^{-1} \ldots, \lambda^{-1}) \backslash 
\overline{\Pi}_{T_r} (1,\lambda^{-1}, \ldots, \lambda^{-1})
\end{equation}
and
$$
\widehat{\bm{\tilde{u}} (T), \ell_{n-t}}
\le
\arctan \left(
\frac{\sqrt{t} \displaystyle{\max_{n-t+1\le i \le n}}  \lambda^{-1}T^{mk_{i}}}{T^{mk_{n-t}}}
\right)
\le\arctan \left(\frac{\sqrt{t}}{\lambda}\right) = \sigma.
$$
So all the constructed  vectors $\bm{u} (T) $
satisfy the condition (\ref{angle}).
Now it turns out to be  sufficient to satisfy the lacunarity
condition 
(\ref{lasa}), for both {\bf Cases 1} and {\bf 2}.

\subsection{Lacunarity}

Let 
\begin{equation}\label{P3}
T_r=R^r,\,\,\,\text{where}\,\,\,
R = 
\left(
2\sqrt{t+1}\,\lambda^t\gamma^{-m} \, \frac{\cos (\omega-\gamma)}{\cos (\omega+\gamma)}
\right)^{\frac{1}{mk_{n-t}}}.
\end{equation}
It is clear that
\begin{equation}\label{P20}
 R 
 >
 \gamma^{-1/k_{n-t}}\lambda^{t/mk_{n-1}}.
\end{equation}

We put $T = T_r $ and define the sequence of vectors $\bm{u}_r\in\mathbb{R}^{n}$
from the equality
$$
\bm{w}_{r}  = (\bm{u}_{r} , \bm{v}_{r} ) = \bm{w}(T_r),
$$
where $\bm{w} (\cdot )$ is defined in the previous subsection. 

\begin{lemma}
The sequence $|\bm{\tilde{u}}_r|_e$   satisfies the lacunarity condition (\ref{lasa}).
\end{lemma}

\begin{proof}[Proof]
From (\ref{2uy},\ref{2uz}) we see that
$$
|\bm{\tilde{u}}_r|_e\le \sqrt{t+1}\, \lambda^t \gamma^{-m} T_r^{mk_{n-t}}.
$$
From (\ref{2uu}) we have the lower bound
$$
|\bm{\tilde{u}}_r|_e\ge  T_r^{mk_{n-t}}.
$$
Now
$$
\frac{|\bm{\tilde{u}}_{r+1}|_e}{|\bm{\tilde{u}}_r|_e}
\ge
\frac{T_{r+1}^{mk_{n-t}}}{\sqrt{t+1}\, \lambda^t \gamma^{-m} T_r^{mk_{n-t}}}
\ge \frac{2\cos (\omega-\gamma)}{\cos (\omega+\gamma)}
$$
by (\ref{P3}), and everything is proved.
\end{proof}

\subsection{Application of transference identity}

Now we establish {\bf Fact A}. Let $x \in N(\U)$. So there exists a constant $c(x)>0$ such that
$$
\|x_1 u_{r,1} +\ldots +x_n u_{r,n} \|>c(x) \qquad (r\geq 1).
$$
We define
$$
\psi_r = \psi (T_r)= \max\limits_{1 \le j \le m} \| \Theta^*_j (\textbf{u}_r)\|.
$$
For any $\bm{q} \in \mathbb{Z}^m_{\neq 0}$, consider the equality
$$
\textbf{u}_r \cdot \bm{x} = \sum\limits_{j =1}^m q_j \Theta^*_j (\textbf{u}_r) -\sum\limits_{i =1}^n (\Theta_i (\textbf{q}) - x_i) u_{r,i}(T_r).
$$
It follows from the triangle inequality that 
\begin{align*}
c(x) &\le \|\textbf{u}_r \cdot \bm{x} \|\\
 &\le m \max\limits_{1 \le j \le m} (\| \Theta^*_j (\textbf{u}_r)\| \cdot |q_j|) + n \max\limits_{1 \le i \le n} (\|\Theta_i (\textbf{q}) - x_i\| \cdot |u_{r,i}|)\\ 
&\le m \psi_r  |\bm{q}| + n \max\limits_{1 \le i \le n} (\|\Theta_i (\textbf{q}) - x_i\| \cdot |u_{r,i}|).
\end{align*}
Here we use the well known inequality $\|a \bm{z}\| \le |a| \|\bm{z}\|$, which holds for all $a \in \mathbb{R}$ and all $\bm{z} \in \mathbb{R}^m$.

It is clear that $\psi_r \rightarrow 0$ as $r \rightarrow \infty$.
We show
that $\psi_r$ is strictly decreasing.
Indeed,
$$
\begin{array}{ll}
\psi_r 
\stackrel{\eqref{P2}}{\ge}
\lambda^{-t/mk_{n-t}}\,\gamma^{1+1/k_{n-t}} T_r^{-1}
&\stackrel{\eqref{P3}}{=}
\gamma T_{r+1}^{-1} \cdot
\left(
R
\lambda^{-t/mk_{n-t}}\,\gamma^{1/k_{n-t}} \right)\\
&\stackrel{\eqref{P20}}{>}
\gamma T_{r+1}^{-1}\\
&\stackrel{\eqref{2.2''}}{\ge}
\psi_{r+1}.
\end{array}
$$
Moreower, from (\ref{2.2''},\ref{P2},\ref{P20}) it follows that
\begin{equation}\label{jj}
 \frac{\psi_{r-1}}{\psi_r}
 \le
 \lambda^{1/mk_{n-t}}
 \gamma^{-1/k_{n-t}} \frac{T_{r}}{T_{r-1}} =
 \lambda^{1/mk_{n-t}}
 \gamma^{-1/k_{n-t}} R \le R^2
 .
\end{equation}
Now we can
choose $r$ in such a way that
\begin{equation}\label{2.9}
\psi_{r-1}  \ge \dfrac{c(x)}{2 m |\bm{q}|}
> \psi_r .
 \end{equation}
 Therefore
$$
c(x) \le n \max\limits_{1 \le i \le n} (\|\Theta_i (\textbf{q}) - x_i\| \cdot |u_{r,i}|) + m |\bm{q}|  \dfrac{c(x)}{2 m |\bm{q}|},
$$
and so
\begin{equation}\label{jjjj}
\dfrac{c(x)}{2 n} \le  \max\limits_{1 \le i \le n} (\|\Theta_i (\textbf{q}) - x_i\| \cdot |u_{r,i}|).
\end{equation}
From (\ref{2.1''},\ref{2uz})
we deduce
$$
|u_{r,i}| \le T_r^{mk_i},\,\,\,\, 1 \le i \le n, i \neq n-t
,
$$
and from  (\ref{2uy})
we have
$$
|u_{r,n-t}| \le \lambda^t\gamma^{-m} T_r^{mk_{n-t}} 
.$$
So in any case
 $$
|u_{r,i}| \le \lambda^t\gamma^{-m} T_r^{mk_i} 
\stackrel{\eqref{2.2''}}{\le}
\lambda^t\gamma^{m(k_i-1)} \psi_r^{-mk_i}
=
\lambda^t\gamma^{m(k_i-1)} \psi_{r-1}^{-mk_i}
\cdot\left(\frac{ \psi_{r-1}}{ \psi_r}\right)^{mk_i}
$$
and by  (\ref{jj}) and the left inequality from 
(\ref{2.9}),
\begin{equation}\label{j*}
|u_{r,i}| \le 
\lambda^t\gamma^{m(k_i-1)}
\left(
\frac{2mR^2}{c(x)}\right)^{mk_i}  |\bm{q}|^{mk_i}
.
\end{equation}
 Now from (\ref{jjjj},\ref{j*}) we get
$$
\max\limits_{1 \le i \le n} (\|\Theta_i (\textbf{q}) - x_i\| \cdot |\bm{q}|^{mk_i}) \ge \kappa
$$
with  some constant $\kappa > 0$, independent of $\bm{q}$.
Since the choice of the vector $\bm{q}$ was arbitrary, we have shown that
 $x\in \Bad_\Theta(\bm{k}, n,m)$. Hence we prove $N(\U) \subset \Bad_\Theta(\bm{k}, n,m)$.


\begin{thebibliography}{99}


\bibitem{An} 
J. An, \emph{Two-dimensional badly approximable vectors and Schmidt's game}, arXiv:1204.3610.

\bibitem{ABV}
J. An, V. Beresnevich, S. Velani, \emph{Badly approximable points on planar curves and winning}, arXiv:1409.0064 (2014). 

\bibitem{BV}
D. Badziahin, S. Velani, \emph{Badly approximable points on planar curves and a problem of Davenport}, Mathematische Annalen. 359 (3) (2014), 969-1023.

\bibitem{BPV}
D. Badziahin, A. Pollington, S. Velani, \emph{On a problem in simultaneous Diophantine approximation: Schmidt's conjecture}, Annals of Mathematics. 174 (2011), 1837-1883.

\bibitem{BM} P. Bengoechea,  N. Moshchevitin, \emph{On weighted twisted badly approximable numbers}, submitted, arXiv:1507.07119.

\bibitem{Ber}
V. Beresnevich, \emph{Badly approximable points on manifolds}, Invent. math. (2015) DOI: 10.1007/s00222-015-0586-8. 

\bibitem{BHKV}
Y. Bugeaud, S. Harrap, S. Kristensen, and S. Velani, \emph {On shrinking targets for $\Z^m$ actions on tori} Mathematika 56 (2010) 193-202



\bibitem{ET} 
M. Einsiedler, J. Tseng, \emph{Badly approximable systems of affine forms, fractals, and Schmidt games}, J. Reine Angew. Math. 660 (2011), 83-97.


\bibitem{GE}
O. N. German and K. G. Evdokimov, \emph{A strengthening of Mahler's transference theorem},  Izvestiya: Mathematics 79:1 (2015) 60-73.

\bibitem{HG}
H.J. Godwin, \emph{On the theorem of Khintchine}, Proc. London Math. Soc. V.3, 1, 1953, 211-221

\bibitem{HM} S. Harrap, N. Moshchevitin, \emph{A note on weighted badly approximable linear forms}, to appear in Glasgow Mathematical Journal.

\bibitem{JC}
J.W.S. Cassels,\emph{ An introduction to Diophantine approximation}, Cambridge
Tracts in Math., vol. 45, Cambridge Univ. Press, Cambridge,
1957.

\bibitem{D}
S. Dani, \emph{On badly approximable numbers, Schmidt games and bounded orbits of flows}, Number theory and dynamical systems (York, 1987), London Math. Soc. Lec-
ture Note Ser., vol. 134, CUP, 1989, 69-86.

\bibitem{J}
V. Jarn\'{i}k,\emph{ O linea\'{r}n\'{i}ch nehomogenn\'{i}ch
  diofantick\'{y}ch aproximac\'{i}ch (on linear inhomogeneous Diophantine approximations)},
   Rozpravy II. T\v{r}\'{i}dy \v{C}esk\'{e} 
   Akad. 51 (1941), no. 29, 21. MR 0021015  
 
\bibitem{J2} 
 V. Jarn\'{i}k,\emph{ Sur les approximations diophantiques lin\'{e}aires non homog\'{e}nes}, Acad. Tch\'{e}que Sci. Bull. Int. Cl. Sci. Math. Nat. 47 (1946), 145-160 (1950). 
 
\bibitem{K} D.H. Kim, \emph{The shrinking target property of irrational rotations}, Nonlinearity 20 (2007), 7, 1637-1643.

\bibitem{KH}
A. Khintchine, \emph{\"Uber eine Klasse linearer diophantischer Approximationen}, Rend. Circ. Mat. Palermo 50 (1926), 170-195.

\bibitem{KH2}
A. Khintchine, 
\emph{\"Uber die angen\"{a}herte Aufl\"{o}sung linearer Gleichungen in ganzen Zahlen},
 Acta Arith. 2 (1937), 161-172.

 \bibitem{KH3}
 A. Y. Khintchine,\emph{ Regular systems of linear equations and a general problem of Chebyshev}, Izvestiya Akad. Nauk SSSR. Ser. Mat. 12 (1948), 249-258 (Russian). 
 
 \bibitem{KL}
D. Kleinbock,\emph{ Badly approximable systems of affine forms}, J. Number Theory 79 (1999), no. 1, 83-102.




\bibitem{NGM}
N. Moshchevitin, \emph{A note on badly approximable affine forms and winning sets}, Mosc. Math. J. 11 (2011), no. 1, 129-137.

\bibitem{PV} A. Pollington, S. Velani, \emph{On simultaneously badly approximable numbers}, J. London Math. Soc. (2) 66 (2002), 29-40.

\bibitem{FSU}
L. Fishman, D. S. Simmons, M. Urba\'{n}ski, 
\emph{Diophantine approximation and the geometry of limit sets in Gromov hyperbolic metric spaces} arXiv:1301.5630v11.

\bibitem{Sch1}
W. M. Schmidt, \emph{Badly approximable systems of linear forms}, J. Number Theory 1 (1969), 139-154. MR 0248090. 

\bibitem{Sch2}
 W. M. Schmidt, \emph{Diophantine approximation}, Lecture Notes in Math., vol. 785, Springer-Verlag, Berlin 1980.
 
\bibitem{Sch3}
  W. M. Schmidt, \emph{On badly approximable numbers and certain games}, Trans. Amer. Math. Soc. 123 (1966), 178-199. MR 0195595. 

\bibitem{Spr}
V. G. Sprindzuk, \emph{Achievements and problems in Diophantine approximation theory}, Russian
Math. Surveys 35 (1980), 1-80.

\bibitem{T} J. Tseng, \emph{Badly approximable affine forms and Schmidt games}, J. Number Theory 129 (2009), 3020-3025.

\end{thebibliography}
\end{document}